\documentclass[10pt,letter,oneside]{amsart} 

\usepackage[voffset = 0pt, hoffset = 0pt, footskip = 20pt,  textheight = 650pt, textwidth = 480pt]{geometry}

\usepackage{amsmath}
\usepackage{amsfonts}
\usepackage{amssymb}
\usepackage{mathtools}
\usepackage{physics}
\usepackage{mathrsfs} 
\usepackage{amsthm} 
\usepackage{tikz}

\usepackage{verbatim}

\usepackage{amsthm}
\numberwithin{equation}{subsection}
\theoremstyle{plain}
\newtheorem{thm}{Theorem}[section]

\newtheorem{lem}[thm]{Lemma}
\newtheorem{cor}[thm]{Corollary}

\theoremstyle{definition}

\newtheorem{ques}[thm]{Question}
\theoremstyle{remark}
\newtheorem{rmk}[thm]{Remark}

\newtheorem{exmp}[thm]{Example}

\begin{document}



\newcommand{\id}{\mathrm{id}}
\newcommand{\into}{\hookrightarrow}
\newcommand{\onto}{\twoheadrightarrow}
\newcommand{\otspam}{\mathbin{\reflectbox{$\mapsto$}}}

\newcommand{\DD}{\mathbb{D}}
\newcommand{\fbar}{\overline{f}}
\newcommand{\gobar}{\overline{\omega}}


\newcommand{\e}{\mathrm{e}}
\newcommand{\xbar}{\bar{x}}
\newcommand{\ybar}{\bar{y}}
\newcommand{\vp}{\mathbf{p}}
\newcommand{\vw}{\mathbold{w}}
\newcommand{\vx}{\mathbf{x}}
\newcommand{\vy}{\mathbf{y}}
\newcommand{\vga}{\mathbold{\alpha}}
\newcommand{\vgb}{\mathbold{\beta}}
\newcommand{\vxi}{\mathbold{\xi}}
\newcommand{\im}{\operatorname{im}}
\newcommand{\coker}{\operatorname{coker}}
\newcommand{\Hom}{\operatorname{Hom}}
\newcommand{\End}{\operatorname{End}}
\newcommand{\res}{\operatorname{res}}
\newcommand{\Ann}{\operatorname{Ann}}
\newcommand{\Spec}{\operatorname{Spec}}
\newcommand{\Proj}{\operatorname{Proj}}
\newcommand{\oo}{\mathscr{O}}
\newcommand{\Pic}{\operatorname{Pic}}
\newcommand{\acton}{\mathbin{\rotatebox[origin=c]{-90}{$\circlearrowleft$}}}
\newcommand{\actedonby}{\mathbin{\rotatebox[origin=c]{90}{$\circlearrowright$}}}
\newcommand{\scr}[1]{\mathscr{#1}}
\newcommand{\obj}{\operatorname{obj}}
\newcommand{\Ad}{\operatorname{Ad}}
\newcommand{\ad}{\operatorname{ad}}
\newcommand{\LgrpGL}{\mathrm{GL}}
\newcommand{\LgrpSL}{\mathrm{SL}}
\newcommand{\LgrpO}{\mathrm{O}}
\newcommand{\LgrpSO}{\mathrm{SO}}
\newcommand{\LgrpSp}{\mathrm{Sp}}
\newcommand{\LgrpU}{\mathrm{U}}
\newcommand{\LgrpSU}{\mathrm{SU}}
\newcommand{\LgrpPGL}{\mathrm{PGL}}
\newcommand{\LgrpPSL}{\mathrm{PSL}}


\newcommand{\ga}{\alpha} 
\newcommand{\gb}{\beta} 
\newcommand{\gc}{\gamma} \newcommand{\gC}{\Gamma}
\newcommand{\gd}{\delta} \newcommand{\gD}{\Delta}
\newcommand{\gve}{\varepsilon} \renewcommand{\ge}{\epsilon} 
\newcommand{\gi}{\iota} 
\newcommand{\gk}{\kappa} 
\newcommand{\gl}{\lambda} \newcommand{\gL}{\Lambda}
\newcommand{\go}{\omega} \newcommand{\gO}{\Omega}
\newcommand{\gvp}{\varphi} 
\newcommand{\gr}{\rho} 
\newcommand{\gs}{\sigma} \newcommand{\gS}{\Sigma}
\newcommand{\gth}{\theta} \newcommand{\gTh}{\Theta}
\newcommand{\gu}{\upsilon} 
\newcommand{\gz}{\zeta} 

\newcommand{\bbA}{\mathbb{A}}
\newcommand{\bbB}{\mathbb{B}}
\newcommand{\bbC}{\mathbb{C}}
\newcommand{\bbD}{\mathbb{D}}
\newcommand{\bbE}{\mathbb{E}}
\newcommand{\bbF}{\mathbb{F}}
\newcommand{\bbG}{\mathbb{G}}
\newcommand{\bbH}{\mathbb{H}}
\newcommand{\bbI}{\mathbb{I}}
\newcommand{\bbJ}{\mathbb{J}}
\newcommand{\bbK}{\mathbb{K}}
\newcommand{\bbL}{\mathbb{L}}
\newcommand{\bbM}{\mathbb{M}}
\newcommand{\bbN}{\mathbb{N}}
\newcommand{\bbO}{\mathbb{O}}
\newcommand{\bbP}{\mathbb{P}}
\newcommand{\bbQ}{\mathbb{Q}}
\newcommand{\bbR}{\mathbb{R}}
\newcommand{\bbS}{\mathbb{S}}
\newcommand{\bbT}{\mathbb{T}}
\newcommand{\bbU}{\mathbb{U}}
\newcommand{\bbV}{\mathbb{V}}
\newcommand{\bbW}{\mathbb{W}}
\newcommand{\bbX}{\mathbb{X}}
\newcommand{\bbY}{\mathbb{Y}}
\newcommand{\bbZ}{\mathbb{Z}}

\newcommand{\calA}{\mathcal{A}}
\newcommand{\calB}{\mathcal{B}}
\newcommand{\calC}{\mathcal{C}}
\newcommand{\calD}{\mathcal{D}}
\newcommand{\calE}{\mathcal{E}}
\newcommand{\calF}{\mathcal{F}}
\newcommand{\calG}{\mathcal{G}}
\newcommand{\calH}{\mathcal{H}}
\newcommand{\calI}{\mathcal{I}}
\newcommand{\calJ}{\mathcal{J}}
\newcommand{\calK}{\mathcal{K}}
\newcommand{\calL}{\mathcal{L}}
\newcommand{\calM}{\mathcal{M}}
\newcommand{\calN}{\mathcal{N}}
\newcommand{\calO}{\mathcal{O}}
\newcommand{\calP}{\mathcal{P}}
\newcommand{\calQ}{\mathcal{Q}}
\newcommand{\calR}{\mathcal{R}}
\newcommand{\calS}{\mathcal{S}}
\newcommand{\calT}{\mathcal{T}}
\newcommand{\calU}{\mathcal{U}}
\newcommand{\calV}{\mathcal{V}}
\newcommand{\calW}{\mathcal{W}}
\newcommand{\calX}{\mathcal{X}}
\newcommand{\calY}{\mathcal{Y}}
\newcommand{\calZ}{\mathcal{Z}}

\newcommand{\scrA}{\mathscr{A}}
\newcommand{\scrB}{\mathscr{B}}
\newcommand{\scrC}{\mathscr{C}}
\newcommand{\scrD}{\mathscr{D}}
\newcommand{\scrE}{\mathscr{E}}
\newcommand{\scrF}{\mathscr{F}}
\newcommand{\scrG}{\mathscr{G}}
\newcommand{\scrH}{\mathscr{H}}
\newcommand{\scrI}{\mathscr{I}}
\newcommand{\scrJ}{\mathscr{J}}
\newcommand{\scrK}{\mathscr{K}}
\newcommand{\scrL}{\mathscr{L}}
\newcommand{\scrM}{\mathscr{M}}
\newcommand{\scrN}{\mathscr{N}}
\newcommand{\scrO}{\mathscr{O}}
\newcommand{\scrP}{\mathscr{P}}
\newcommand{\scrQ}{\mathscr{Q}}
\newcommand{\scrR}{\mathscr{R}}
\newcommand{\scrS}{\mathscr{S}}
\newcommand{\scrT}{\mathscr{T}}
\newcommand{\scrU}{\mathscr{U}}
\newcommand{\scrV}{\mathscr{V}}
\newcommand{\scrW}{\mathscr{W}}
\newcommand{\scrX}{\mathscr{X}}
\newcommand{\scrY}{\mathscr{Y}}
\newcommand{\scrZ}{\mathscr{Z}}

\newcommand{\frA}{\mathfrak{A}}
\newcommand{\frB}{\mathfrak{B}}
\newcommand{\frC}{\mathfrak{C}}
\newcommand{\frD}{\mathfrak{D}}
\newcommand{\frE}{\mathfrak{E}}
\newcommand{\frF}{\mathfrak{F}}
\newcommand{\frG}{\mathfrak{G}}
\newcommand{\frH}{\mathfrak{H}}
\newcommand{\frI}{\mathfrak{I}}
\newcommand{\frJ}{\mathfrak{J}}
\newcommand{\frK}{\mathfrak{K}}
\newcommand{\frL}{\mathfrak{L}}
\newcommand{\frM}{\mathfrak{M}}
\newcommand{\frN}{\mathfrak{N}}
\newcommand{\frO}{\mathfrak{O}}
\newcommand{\frP}{\mathfrak{P}}
\newcommand{\frQ}{\mathfrak{Q}}
\newcommand{\frR}{\mathfrak{R}}
\newcommand{\frS}{\mathfrak{S}}
\newcommand{\frT}{\mathfrak{T}}
\newcommand{\frU}{\mathfrak{U}}
\newcommand{\frV}{\mathfrak{V}}
\newcommand{\frW}{\mathfrak{W}}
\newcommand{\frX}{\mathfrak{X}}
\newcommand{\frY}{\mathfrak{Y}}
\newcommand{\frZ}{\mathfrak{Z}}
\newcommand{\fra}{\mathfrak{a}}
\newcommand{\frb}{\mathfrak{b}}
\newcommand{\frc}{\mathfrak{c}}
\newcommand{\frd}{\mathfrak{d}}
\newcommand{\fre}{\mathfrak{e}}
\newcommand{\frf}{\mathfrak{f}}
\newcommand{\frg}{\mathfrak{g}}
\newcommand{\frh}{\mathfrak{h}}
\newcommand{\fri}{\mathfrak{i}}
\newcommand{\frj}{\mathfrak{j}}
\newcommand{\frk}{\mathfrak{k}}
\newcommand{\frl}{\mathfrak{l}}
\newcommand{\frm}{\mathfrak{m}}
\newcommand{\frn}{\mathfrak{n}}
\newcommand{\fro}{\mathfrak{o}}
\newcommand{\frp}{\mathfrak{p}}
\newcommand{\frq}{\mathfrak{q}}
\newcommand{\frr}{\mathfrak{r}}
\newcommand{\frs}{\mathfrak{s}}
\newcommand{\frt}{\mathfrak{t}}
\newcommand{\fru}{\mathfrak{u}}
\newcommand{\frv}{\mathfrak{v}}
\newcommand{\frw}{\mathfrak{w}}
\newcommand{\frx}{\mathfrak{x}}
\newcommand{\fry}{\mathfrak{y}}
\newcommand{\frz}{\mathfrak{z}}
\newcommand{\frgl}{\mathfrak{gl}}
\newcommand{\frsl}{\mathfrak{sl}}
\newcommand{\frso}{\mathfrak{so}}
\newcommand{\frsp}{\mathfrak{sp}}
\newcommand{\frsu}{\mathfrak{su}}

\title[]{Stability of the canonical extension of tangent bundles on Picard-rank-1 Fano varieties}
\author{Kuang-Yu Wu}
\address{Department of Mathematics, Statistics, and Computer Science, University of Illinois at Chicago, Chicago, IL 60607}
\email{kwu33@uic.edu}


\newcommand{\Ext}{\operatorname{Ext}}

\theoremstyle{plain}
\newtheorem*{thmA}{Theorem A}
\newtheorem*{corB}{Corollary B}

\begin{abstract}
	We consider slope stability of the canonical extension of the tangent bundle by the trivial line bundle and with the extension class $c_1(\calT_X)$ on Picard-rank-$1$ Fano varieties.
	In cases where the index divides the dimension or the dimension plus one,
	we show that stability of the tangent bundle implies (semi)stability of the canonical extension.
	One consequence of our result is that the canonical extensions on moduli spaces of stable vector bundles with a fixed determinant on a curve are at least semistable,
	and stable in some cases.
\end{abstract}

\maketitle

\section{Introduction}

We work throughout over $\bbC$.
Let $X$ be a smooth Fano variety.
If $X$ admits a K\"ahler-Einstein metric,
it is well known that the tangent bundle $\calT_X$ of $X$
is slope polystable with respect to the anticanonical line bundle $-K_X$
by the Kobayashi-Hitchin correspondence \cite{Kob82,Lub83}.
In addition to this,
Tian showed in \cite{Tia92} that the canonical extension of the tangent bundle $\calT_X$ by the trivial line bundle $\calO_X$
$$
0 \to \calO_X \to \calE_X \to \calT_X \to 0
$$
defined by the extension class $c_1(X)$
is also slope polystable with respect to $-K_X$.
See also \cite{Li21,DGP20} for generalizations of this result.

The first Chern class $c_1(X)$ of $X$ lies in the $(1,1)$-Dolbeault cohomology $H^{1,1}(X)$,
which is isomorphic to the sheaf cohomology $H^1(X,\gO^1_X)$ by Dolbeault's theorem.
Since we have
$$
H^1 (X,\gO^1_X) \cong \operatorname{Ext}^1(\calO_X,\gO^1_X) \cong \operatorname{Ext}^1(\calT_X,\calO_X)
$$
by for example \cite[Propositions 6.3 and 6.7]{Har77},
$c_1(X)$ can be viewed as an extension class in $\operatorname{Ext}^1(\calT_X,\calO_X)$.
We will call the extension $\calE_X$ defined by $c_1(X)$ the \textit{canonical extension} on $X$.
One example is the Euler sequence
$$
0 \to \calO_{\bbP^n} \to \calO_{\bbP^n} (1)^{\oplus (n + 1)} \to \calT_{\bbP^n} \to 0
$$
on projective spaces $\bbP^n$.

By slope stability,
we mean the following.
Fix an ample line bundle $\calL$ on $X$,
and let $n := \dim(X)$.
Given a torsion free coherent sheaf $\calF$ on $X$,
define its \textit{slope} with respect to $\calL$ by
$$
\mu_{\calL} (\calF) := \frac{c_1(\calF) \cdot c_1(\calL)^{n - 1}}{\rank(\calF)}
$$
Then we call $\calF$ 
\begin{itemize}
	\item \textit{slope stable} if $\mu_{\calL} (\calG) < \mu_{\calL} (\calF)$
	for all non-zero proper subsheaves $0 \neq \calG \subsetneq \calF$,
	
	\item \textit{slope semistable} if $\mu_{\calL} (\calG) \leq \mu_{\calL} (\calF)$
	for all non-zero proper subsheaves $0 \neq \calG \subsetneq \calF$, and
	
	\item \textit{slope polystable} if $\calF$ is a direct sum of stable sheaves of the same slope.
\end{itemize}
It is straightforward from the definition
that slope stability implies slope polystability
and that slope polystability implies slope semistability.

\subsection{Main result and idea of the proof}

The \textit{index} of a smooth Fano variety $X$ is defined to be the largest integer $i$
such that $-K_X = \calL^{\otimes i}$ for some ample line bundle $\calL$ on $X$.

Our main result is the following.

\begin{thmA}[= Theorem \ref{thm:stabofE_X}]
	Let $X$ be a smooth Picard-rank-$1$ Fano variety of dimension $n$ and index $i$.
	\begin{itemize}
		\item[(a).]
		Suppose $i$ divides $n$
		and that the tangent bundle $\calT_X$ is slope stable with respect to $-K_X$.
		Then the canonical extension $\calE_X$ is slope stable with respect to $-K_X$.
		
		\item[(b).]
		Suppose $i$ divides $n+1$
		and that the tangent bundle $\calT_X$ is slope semistable with respect to $-K_X$.
		Then the canonical extension $\calE_X$ is slope semistable with respect to $-K_X$.
	\end{itemize}
\end{thmA}

The idea of our proof is to show that
given a destabilizing subsheaf $\calF \subset \calE_X$,
its image in $\calT_X$ is also destabilizing.
The key lemma (Lemma \ref{lem:subbdlofE_X}) is that
all proper subsheaves of $\calE_X$ have strictly lower degrees than $\calE_X$,
which we prove by using various vanishing theorems.
We remark that Lemma \ref{lem:subbdlofE_X} is analogous to \cite[Theorem 3]{Rei77}
(see also \cite[Proposition 2.2]{PW95} and \cite[Lemma 2.2]{Iye14}),
and that similarly it has the following implication,
of which a different proof can be found in \cite[Corollary 6.5]{GKP22}.

\begin{corB}[= Corollary \ref{cor:i=1}]
	Let $X$ be a smooth Picard-rank-$1$ Fano variety of dimension $n$ and index $1$.
	Then the canonical extension $\calE_X$ is slope stable with respect to $-K_X$.
\end{corB}

Our result partially answers the following question raised in \cite{GKP22}.

\begin{ques} \cite[Problem 6.6]{GKP22}
	Let $X$ be a smooth Picard-rank-$1$ Fano variety
	whose tangent bundle $\calT_X$ is slope (semi)stable with respect to $-K_X$.
	Is the canonical extension $\calE_X$ slope (semi)stable with respect to $-K_X$ as well?
\end{ques}

The following example shows that the Picard-rank-$1$ assumption is necessary in our results.

\begin{exmp}
	Let $\gvp : X \to \bbP^2$ be the blow-up of $\bbP^2$ at a point.
	The Picard group of $X$ has rank $2$
	and is generated by $H := \gvp^*\calO_{\bbP^2}(1)$ and the exceptional divisor class $E$,
	with the intersection product given by $H^2 = -E^2 = 1$ and $H.E = 0$.
	The canonical line bundle of $X$ is given by $K_X = \calO_X (- 3H + E)$,
	and hence $X$ is a Fano surface of index $1$.
	
	The tangent bundle $\calT_X$ on $X$ is slope semistable (but not stable) with respect to $-K_X$ \cite{Fah89}.
	However,
	the canonical extension $\calE_X$ on $X$ is not slope semistable with respect to $-K_X$ \cite[Theorem 3.2]{Tia92}.
	In fact,
	$\calE_X$ contains $\calO_X(H)$ as a destabilizing subsheaf;
	as one may check,
	$$
	\mu_{-K_X} (\calO_X(H)) = \frac{H.(3H - E)}{1} = 3 > \frac{8}{3} = \frac{(3H - E)^2}{3} = \mu_{-K_X} (\calE_X) \, .
	$$
	
\end{exmp}

\begin{rmk}
Since $X$ is a toric variety,
another way to see that $\calT_X$ is slope semistable but $\calE_X$ is not
is to check that the corresponding polytope $P_X$ of $X$
satisfies the ``subspace concentration condition'' but not the ``affine subspace concentration condition'' \cite{HNS22,Wu22}.
\end{rmk}

Although the assumption of slope (semi-)stability of the tangent bundle $\calT_X$ might seem strong,
there are a number of Picard-rank-$1$ Fano varieties that have been proved to have slope stable tangent bundles;
see Section \ref{sec:app} and \cite{Sub91,PW95,Hwa98,HM99,Hwa00,BS05,Bis10,Iye14,Liu18,Kan21}.
In fact, it was conjectured that all Picard-rank-$1$ Fano varieties have slope (semi)stable tangent bundles
until counterexamples were found by \cite{Kan21}.
See also \cite{Fah89,Tia92,Ste96} for related works on stability of tangent bundles on Fano varieties,
and \cite{DDK19,BDGP21,HNS22} for works on stability of tangent bundles on Fano toric varieties.

The canonical extensions and their slope stability does not seem to be as well-studied in the literature to the best of our knowledge.
See \cite{GKP22} for related works on slope stability of the canonical extension on Fano varieties,
and \cite{Wu22} for related works on slope stability of the canonical extension on Fano toric varieties.
See also \cite{GW20,HP21,Iwa22,Mul22} for other works on the canonical extensions.

\subsection{Applications of the main result} \label{sec:app}

\subsubsection{Moduli spaces of stable vector bundles with a fixed determinant}

As an application of our result,
we consider the moduli space $M := \calS \calU_C (r , \eta)$ of stable vector bundles of rank $r$ with a fixed determinant $\eta$ of degree $d$ on a curve $C$ of genus $g \geq 3$.
Assume that $r$ is coprime to $d$.
In this case,
$M$ is a smooth Picard-rank-$1$ Fano variety of dimension $n = (r^2 - 1) (g - 1)$ and index $2$ \cite{Ran73}.

It was proved in \cite{Hwa00} that the tangent bundle $\calT_M$ on $M$
is slope stable with respect to $-K_M$
if $r = 2$ (and $d$ is odd).
This result was later generalized in \cite{Iye14} to arbitrary rank $r$.
Thus,
our result implies the following.

\begin{cor}
	The canonical extension $\calE_M$ on $M := \calS \calU_C (r , \eta)$
	is slope semistable with respect to $-K_M$.
	If in addition at least one of $g$ and $r$ is odd,
	then $\calE_M$ is slope stable with respect to $-K_M$.
\end{cor}

This provides an extra evidence that
these moduli spaces might admit K\"ahler-Einstein metric,
although whether this is true remains an open question.

\subsubsection{Low-dimensional Picard-rank-1 Fano varieties}

Let $X$ be a smooth Picard-rank-$1$ Fano variety of dimension $n$ and index $i$.
We consider the cases where $n = 2 , 3$.

By a well-known result of Kobayashi-Ochiai \cite{KO72},
the index $i$ is at most $n + 1$,
and $X \cong \bbP^n$ if $i$ is exactly $n + 1$.
The canonical extension on $\bbP^n$ is $\calO_{\bbP^n} (1)^{\oplus (n + 1)}$,
which is slope polystable.
Also,
if $i = 1$,
then the canonical extension $\calE_X$ is slope stable by Corollary B.
Thus,
we may assume $1 < i < n + 1$.

Then,
the possible values for the index are $i = 2$ if $n = 2$,
and $i = 2 , 3$ if $n = 3$.
One can check the hypotheses we need,
namely $i$ divides $n$ or $n + 1$,
holds in all cases.
Furthermore,
by \cite[Theorem 3(1)]{PW95}
the tangent bundle $\calT_X$ is slope stable since $i \geq n - 1$ in all cases.
Our main theorem then applies.

\begin{cor}
	Let $X$ be a smooth Picard-rank-$1$ Fano variety of dimension $n \leq 3$.
	Then $\calE_X$ is slope semistable with respect to $-K_X$.
\end{cor}

In particular,
we give a different proof for \cite[Corollary 6.16]{GKP22}.

\subsection{Acknowledgment}

{
	I would like to thank my advisor Julius Ross for setting this project
	and for discussions about this project.
	I am also grateful to Yeqin Liu for fruitful conversations.
	Moreover,
	I would like to thank Daniel Greb for bringing their article to our attention and his advice.
}

\section{Proof of the main result} \label{sec:pfofmainresult}

Let $X$ be a smooth Fano variety of dimension $n$.
By this we mean that
the anticanonical line bundle $-K_X$ on $X$ is ample.
Define the \textit{index} of $X$ as
the largest integer $i$ such that $-K_X = \calL^{\otimes i}$ for some ample line bundle $\calL$ on $X$.

We define the canonical extension $\calE_X$ on $X$ as the extension
\begin{equation} \label{eqn:E_X}
	0 \to \calO_X \to \calE_X \to \calT_X \to 0.
	\tag{$\dagger$}
\end{equation}
defined by the extension class $c_1(X) \in H^{1,1}(X) \cong H^1(X,\gO^1_X) \cong \operatorname{Ext}^1(\calT_X,\calO_X)$.

We will assume $X$ has Picard rank one
and let $\calL \in \Pic X$ be the ample generator of $\Pic (X)$.
Then,
by the \textit{degree} of a coherent sheaf $\calF$ on $X$,
we mean the integer $\deg (\calF)$ such that $c_1(\calF) = \deg (\calF) \cdot c_1(\calL)$.
Clearly we have $\deg (K_X) = -i$
and $\deg (\calT_X) = \deg (\calE_X) = i$.

To study stability of the canonical extension $\calE$,
we consider its proper subsheaves.
The following lemma states that their degrees are all smaller than the index $i$.

\begin{lem} \label{lem:subbdlofE_X}
	Let $\calF \subsetneq \calE_X$ be a proper subsheaf.
	Then $\deg (\calF) < i$.
\end{lem}

Before giving the proof,
we note the following immediate consequence of the lemma.

\begin{cor} \label{cor:i=1}
	Let $X$ be a smooth Picard-rank-$1$ Fano variety of dimension $n$ and index $1$.
	Then the canonical extension $\calE_X$ is slope stable with respect to $-K_X$.
\end{cor}

\begin{proof}
	Since the index is $1$,
	by the lemma,
	all proper subsheaves $\calF \subsetneq \calE_X$ has non-positive degrees.
	Consequently,
	their slopes are also non-positive
	and hence must be less than the slope of $\calE_X$,
	which is positive.
\end{proof}

Now we prove Lemma \ref{lem:subbdlofE_X}.

\begin{proof}[Proof of Lemma \ref{lem:subbdlofE_X}.]
	Write $d := \deg (\calF)$.
	Arguing by contradiction,
	suppose that $d \geq i$.
	Let $\calF'$ be the dual of the cokernel of the inclusion $\calF \subset \calE_X$.
	Then $\calF'$ is a subsheaf of the dual $\calE_X^{\vee}$ of $\calE_X$.
	Note that $\calF'$ is torsion free since $\calE_X^{\vee}$ is.
	Let $d'$ be the degree of $\calF'$.
	We have $d' = d - i \geq 0$.
	
	From here,
	we divide matters into three cases according to the rank and the degree of $\calF'$.
	
	\textit{Case 1.}
	If the rank of $\calF'$ is greater than $1$,
	write $r := \rank (\calF')$.
	The inclusion $\calF' \subset \calE_X^{\vee}$ induces
	an inclusion $\det(\calF') = \calL^{d'} \subseteq \bigwedge \nolimits ^{\!r} \calE_X^{\vee}$
	and hence a nonzero element in $\Hom(\calL^{d'} , \bigwedge \nolimits ^{\!r} \calE_X^{\vee}) \cong H^0 (X , \bigwedge \nolimits ^{\! r} \calE_X^{\vee} \otimes \calL^{-d'})$.
	The claim is that this cohomology group is in fact zero,
	which gives a contradiction.
	
	To see this,
	note that the short exact sequence (\ref{eqn:E_X}) defining $\calE_X$ induces the following short exact sequence
	$$
	0 \to \calO_X \otimes \bigwedge \nolimits ^{\! r-1} \calT_X \to \bigwedge \nolimits ^{\! r} \calE_X \to \bigwedge \nolimits ^{\! r} \calT_X \to 0 \, ,
	$$
	where the third arrow $\bigwedge \nolimits ^{\! r} \calE_X \to \bigwedge \nolimits ^{\! r} \calT_X$ is induced by $\calE_X \to \calT_X$.
	Dualize this short exact sequence,
	tensor it with $\calL^{-d'}$,
	and take the long exact sequence of cohomology groups.
	We obtain
	$$
	0 \to H^0 (X , \gO_X^r \otimes \calL^{-d'}) \to H^0 \left(X , \bigwedge \nolimits ^{\! r} \calE_X^{\vee} \otimes \calL^{-d'} \right) \to H^0 (X , \gO_X^{r-1} \otimes \calL^{-d'}) \to \cdots
	$$
	We are done if the cohomology groups on the two sides are zero.
	Thus, we want $H^0 (X , \gO_X^k \otimes \calL^{-d'}) = 0$ for all $0 < k \leq n$.
	
	For $k = n$,
	we have
	$H^0 (X , \gO_X^n \otimes \calL^{-d'}) = H^0 (X , \calL^{-i-d'}) = 0$
	by Kodaira vanishing theorem.
	
	For $0 < k < n$, 
	following the proof of \cite[Lemma 2.2]{Iye14},
	we refer to the following vanishing theorems.
	\begin{itemize}
		\item[(1).]
		(Nakano vanishing theorem, see e.g. \cite[Theorem 4.2.3]{Laz04}).
		Let $Y$ be a smooth projective variety of dimension $n$ over $\bbC$,
		and $A$ an ample divisor on $X$.
		Then $H^p (Y , \gO_Y^q (-A)) = 0$ for $p + q < n$.
		\item[(2).]
		\cite[Corollary 3.8]{Kol96}
		Let $Y$ be a smooth projective variety over an algebraically closed field.
		Assume that $Y$ is separably rationally connected.
		Then $H^0 (Y , (\gO_Y)^{\otimes m}) = 0$ for all $m > 0$.
		In particular,
		$H^0 (Y , \gO_Y^m) = 0$ for all $m > 0$
	\end{itemize}
	
	Note that Fano varieties are rationally connected \cite{KMM92},
	and that rational connectedness is equivalent to separable rational connectedness in characteristic $0$ \cite[Proposition 3.3]{Kol96}.
	Therefore,
	these vanishing theorems give us that
	$H^0 (X , \gO_X^k \otimes \calL^{-d'}) = 0$
	for all $0 < k < n$,
	which concludes \textit{Case 1}.
	
	\textit{Case 2.}
	If $\rank (\calF') = 1$ and $d' > 0$,
	we have $\calF' = \calL^{d'}$.
	Note that $\calE_X^{\vee}$ fits in the following short exact sequence
	obtained by dualizing the short exact sequence (\ref{eqn:E_X}) defining $\calE_X$.
	$$
	0 \to \gO_X^1 \to \calE_X^{\vee} \to \calO_X \to 0
	$$
	Compose the inclusion $\calL^{d'} \subset \calE_X^{\vee}$ with the third arrow $\calE_X^{\vee} \to \calO$.
	This composition is necessarily zero,
	since $\Hom (\calL^{d'} , \calO_X) \cong H^0(X , \calL^{-d'}) = 0$
	by Kodaira vanishing theorem.
	This implies that the inclusion $\calL^{d'} \subset \calE_X^{\vee}$ factors through the map $\gO_X^1 \to \calE_X^{\vee}$,
	and hence $\calL^{d'}$ is a subsheaf of $\gO_X^1$.
	However,
	by Nakano vanishing theorem ((1) in \textit{Case 1}),
	we have $\Hom (\calL^{d'} , \gO_X^1) \cong H^0(X , \gO_X^1 \otimes \calL^{-d'}) = 0$.
	Therefore we get a contradiction.
	
	\textit{Case 3.}
	If $\rank (\calF') = 1$ and $d' = 0$,
	we have $\calF' = \calO_X$.
	Similarly to \textit{Case 2},
	compose the inclusion $\calO_X \subset \calE_X^{\vee}$ with the map $\calE_X^{\vee} \to \calO_X$.
	Note that $\Hom (\calO_X , \calO_X) \cong H^0 (X , \calO_X) \cong \bbZ$.
	Thus,
	this composition is either zero or an isomorphism.
	However,
	if it is zero,
	the inclusion $\calO_X \subset \calE_X^{\vee}$ factors through the map $\gO_X^1 \to \calE_X^{\vee}$.
	We get that $\calO_X$ is a subsheaf of $\gO_X^1$.
	This contradicts the vashing theorem (2) in \textit{Case 1},
	which says that $H^0 (X , \gO_X^1) = 0$.
	On the other hand,
	if the composition is an isomorphism,
	we get that the short exact sequence (\ref{eqn:E_X}) splits.
	This is again a contradiction,
	since $\calE_X$ is a non-trivial extension.
\end{proof}

With this lemma,
we prove our main result.

\begin{thm} \label{thm:stabofE_X}
	Let $X$ be a smooth Picard-rank-$1$ Fano variety of dimension $n$ and index $i$.
	\begin{itemize}
		\item[(a).]
		Suppose $i$ divides $n$
		and that the tangent bundle $\calT_X$ is slope stable with respect to $-K_X$.
		Then the canonical extension $\calE_X$ is slope stable with respect to $-K_X$.
		
		\item[(b).]
		Suppose $i$ divides $n+1$
		and that the tangent bundle $\calT_X$ is slope semistable with respect to $-K_X$.
		Then the canonical extension $\calE_X$ is slope semistable with respect to $-K_X$.
	\end{itemize}
\end{thm}

\begin{proof}
	For simplicity, define
	$$
	\mu(\calF) := \frac{\deg (\calF)}{\rank (\calF)}
	$$
	for torsion-free sheaves $\calF$ on $X$.
	It is easy to see that $\mu$ is proportional to the slope $\mu_{-K_X}$ with respect to $-K_X$ in the sense that
	$$
	\mu_{-K_X} (\calF)
	= \frac{c_1(\calF) \cdot c_1(-K_X)^{n - 1}}{\rank(\calF)}
	= \mu (\calF) \cdot \big( i^{n-1} \cdot c_1(\calL)^n \big) .
	$$
	The number $i^{n-1} \cdot c_1(\calL)^n$ is positive,
	so we may replace $\mu_{-K_X}$ by $\mu$ in the definition of slope stability with respect to $-K_X$.
	
	We first prove part (a).
	Suppose the canonical extension $\calE_X$ is not slope stable.
	Then there exists a destablizing subsheaf $\calF \subset \calE_X$ with
	$$
	\mu(\calF) \geq \mu(\calE_X) \, .
	$$
	Recall that $\deg (\calE_X) = i$
	and $\rank (\calE_X) = n + 1$.
	Let $d := \deg (\calF)$ and $r := \rank (\calF)$.
	Then,
	we may rewrite this inequality as
	$$
	\frac{d}{r} \geq \frac{i}{n+1} \,.
	$$
	This implies
	$$
	r \leq \frac{d (n+1)}{i} = \frac{dn}{i} + \frac{d}{i} \,.
	$$
	Note that $dn/i$ is an integer since we assume $i$ divides $n$,
	and that $0 < d/i < 1$ by Lemma \ref{lem:subbdlofE_X}.
	The rank $r$ of $\calF$ is also an integer,
	so we must have
	\begin{equation} \label{eqn:r}
		r \leq \frac{dn}{i} \,. \tag{$*$}
	\end{equation}
	Now,
	consider the image $\gb (\calF) \subset \calT_X$ of $\calF$ under the map $\gb : \calE_X \to \calT_X$.
	Let $\calK$ be the kernel of $\gb|_{\calF}$.
	We get the following short exact sequence.
	$$
	0 \to \calK \to \calF \to \gb (\calF) \to 0
	$$
	The kernel $\calK$ is a subsheaf of $\calO_X$.
	By Kodaira vanishing theorem,
	$\calK$ is either $0$ or $\calL^{-a}$ for some $a \geq 0$.
	Thus,
	we have either $\mu \big( \gb (\calF) \big) = \mu (\calF)$ or
	$$
	\mu \big( \gb (\calF) \big)
	= \frac{\deg (\calF) - (-a)}{\rank (\calF) - 1}
	> \frac{\deg (\calF)}{\rank (\calF)}
	= \mu (\calF)
	$$
	In particular,
	we get $\mu \big( \gb (\calF) \big) \geq \mu (\calF)$.
	Then,
	combining this with the inequality (\ref{eqn:r}),
	we obtain
	$$
	\mu \big( \gb (\calF) \big) \geq \mu (\calF) = \frac{d}{r} \geq \frac{i}{n} = \mu (\calT_X) \, .
	$$
	Therefore,
	$\calT_X$ is not slope stable,
	and we have proved part (a).
	
	The proof for part (b) is similar.
	Suppose the canonical extension $\calE_X$ is not slope semistable.
	Then there exists a destablizing subsheaf $\calF \subset \calE_X$ with $\mu(\calF) > \mu(\calE_X)$,
	which gives the inequality
	$$
	r < \frac{d (n+1)}{i} \, .
	$$
	Here,
	$d(n+1)/i$ is an integer since we assume $i$ divides $n+1$,
	so we have
	$$
	r \leq \frac{d (n+1)}{i} - 1 < \frac{dn}{i} \,.
	$$
	Then the image $\gb (\calF) \subset \calT_X$ satisfies
	$$
	\mu \big( \gb (\calF) \big) \geq \mu (\calF) = \frac{d}{r} > \frac{i}{n} = \mu (\calT_X),
	$$
	and hence $\calT_X$ is not slope semistable.
\end{proof}

\bibliographystyle{ytamsalpha}
\bibliography{ref}

\end{document}